\documentclass[12pt]{amsart}

\usepackage{amsaddr}
\usepackage{amsmath}
\usepackage{amssymb}
\usepackage{amscd}
\usepackage{latexsym}
\usepackage{amsthm}
\usepackage{mathrsfs}
\usepackage{color}
\usepackage{amsfonts}
\usepackage{enumitem}
\usepackage{array,tabularx}
\usepackage{caption,setspace}
\usepackage{mathtools}
\usepackage{tikz}
\usepackage{tikz-cd}
\usetikzlibrary{arrows}
\usetikzlibrary{knots}
\usepackage{hyperref}

\AtBeginDocument{%
	\def\MR#1{}
}

\usepackage{stmaryrd,graphicx}
\newtheorem{thm}{Theorem}[section]
\newtheorem*{thm*}{Theorem}

\newtheorem*{cor*}{Corollary}
\newtheorem{lem}[thm]{Lemma}
\newtheorem*{lem*}{Lemma}
\newtheorem{prop}[thm]{Proposition}
\newtheorem*{prop*}{Proposition}

\theoremstyle{definition}

\newtheorem*{defn*}{Definition}
\newtheorem{conjecture}{Conjecture}

\theoremstyle{remark}
\newtheorem{rem}{Remark}[section]
\newtheorem*{rem*}{Remark}

\newtheorem*{problem*}{Problem}

\newcommand{\QQ}{\mathbb Q}

\newcommand{\CC}{\mathbb C}

\newcommand{\ZZ}{\mathbb Z}

\newcommand{\cC}{\mathcal C}

\newcommand{\PP}{\mathbb P}

\newcommand{\E}{\mathscr E}

\newcommand{\GF}{\mathrm{GF}}

\newcommand{\GL}{\mathrm{GL}}

\newcommand{\Aut}{\mathrm{Aut}}

\renewcommand{\Im}{\mathrm{Im}\;}
\renewcommand{\Re}{\mathrm{Re}\;}

\title{Rationality proofs by curve counting}
\author{Anton Mellit}
\email{anton.mellit@univie.ac.at}
\address{Faculty of Mathematics, University of Vienna, \\
	Oskar-Morgenstern-Platz 1, 1090 Vienna, Austria}
\begin{document}
\onehalfspacing

\begin{abstract}
We propose an approach for showing rationality of an algebraic variety $X$. We try to cover $X$ by rational curves of certain type and count how many curves pass through a generic point. If the answer is $1$, then we can sometimes reduce the question of rationality of $X$ to the question of rationality of a closed subvariety of $X$. This approach is applied to the case of the so-called Ueno-Campana manifolds. Assuming certain conjectures on curve counting, we show that the previously open cases $X_{4,6}$ and $X_{5,6}$ are both rational. Our conjectures are evidenced by computer experiments. In an unexpected twist, existence of lattices $D_6$, $E_8$ and $\Lambda_{10}$ turns out to be crucial.
\end{abstract}

\maketitle

\section{Introduction} 
In November 2014 F. Catanese gave a talk at ICTP,
Trieste about Ueno-Campana varieties.  In particular he spoke about the
following open problem. Let $E$ be the elliptic curve over $\CC$ with complex
multiplication by $\frac{1+\sqrt{-3}}2$ or the curve with complex multiplication by $\sqrt{-1}$. Let $\Gamma\simeq \ZZ/c\ZZ$ be the
group of automorphisms of $E$ or its subgroup with $c\geq 3$. So we have\footnote{We include the case $c=3$ for completeness and because it helps to illustrate our techniques.} $c=3$, $c=4$ or $c=6$ and $c$ determines $E$ uniquely. Let $X_{n,c}=E^n/\Gamma$, the quotient of $E^n$ by the
diagonal action of $\Gamma$. It is well-known that $E^n/\Gamma$ is rational for $n=1,2$. Ueno first studied these varieties in \cite{ueno1975classification} and showed that $E^n/\Gamma$ cannot be rational for $n\geq c=|\Gamma|$. Campana asked (\cite{campana2011remarks})

\begin{problem*}
For which $c$, $n$ is $X_{n,c}$ rational?
\end{problem*}

An introduction to the problem and the state of the art is given in \cite{catanese2015unirationality}. In particular, unirationality of $X_{3,4}$ was proved in \cite{catanese2013unirationality}. Then rationality of $X_{3,4}$ was proved in \cite{colliot2015rationality}. Rationality of $X_{3,6}$ was proved in \cite{oguiso2013explicit}. Then \cite{catanese2015unirationality} established unirationality of $X_{4,6}$. Rationality of $X_{4,6}$ and unirationality of $X_{5,6}$ are still open.

In this paper we give evidence towards rationality of $X_{4,6}$ and $X_{5,6}$. Below we will explain certain curve counting problem. We could only solve this problem by a certain computer-based heuristic approach and our answer is not rigorously justified. So we formulate results of these computations as Conjectures \ref{conj:X46}, \ref{conj:X56 1}, \ref{conj:X56 2}.

\begin{thm*}
    If Conjecture \ref{conj:X46} is true, then $X_{4,6}$ is rational.
\end{thm*}
\begin{thm*}
    If Conjecture \ref{conj:X56 1} is true, then $X_{5,6}$ is unirational. If moreover Conjecture \ref{conj:X56 2} is true and $X_{4,6}$ is rational, then $X_{5,6}$ is rational.
\end{thm*}

As a summary of all the known results we conclude:

\begin{cor*}\label{cor:main cor}
	Suppose Conjectures \ref{conj:X46}, \ref{conj:X56 1}, \ref{conj:X56 2} are true. Let $E$ be an elliptic curve over $\CC$ and let $\Gamma$ be a 
    subgroup of the automorphism group of $E$. Let $n$ be an integer such that $0<n<|\Gamma|$. Then the
    quotient of $E^n$ by the diagonal action of $\Gamma$ is rational.
\end{cor*}

It would be interesting to try to apply our methods to some other abelian varieties or other group actions.

Hopefully, the corresponding curve counting can be achieved by some clever enumerative geometry techniques. This would turn our ``heuristic proofs'' into real proofs.

\section{The main idea}
Mori program teaches us that birational properties of varieties are very much controlled by rational curves on them. Let us try to be not too precise and make a guess, how existence of curves (or rather families of curves) would prove rationality of $X_{5,6}$ for us?  It would be a good situation if some family of rational curves $\{C_s\}_{s\in S}$ existed such that the base $S$ is rational and such that exactly one curve passes through a generic point of $X_{5,6}$. It turns out that just having the latter property is enough for establishing \emph{unirationality} of $X_{5,6}$. To see this, consider an embedding $\iota:X_{4,6}\hookrightarrow X_{5,6}$. If through a generic point of the image of $\iota$ we have exactly one curve from our family, we are done, because then the curves can be parametrized by $X_{4,6}$, so we obtain a dominant rational map $X_{4,6}\rightarrow S$ and unirationality of $X_{5,6}$ as a consequence. Now notice that the union of images of all embeddings $X_{4,6}\hookrightarrow X_{5,6}$ is Zariski dense in $X_{5,6}$, so $\iota$ with the required property exists. A more careful analysis leads to the following lemma:
\begin{lem}\label{lem: main lemma}
Let $X$ be an irreducible algebraic variety of dimension $n$ over $\CC$, and let $\cC=\{C_s\}_{s\in S}$ be an algebraic family of rational curves in $X$. Suppose for a generic point $x\in X$ there is exactly one curve from $\cC$ containing $x$. Let $Z\subset X$ be an irreducible closed subvariety of dimension $n-1$ such that for a generic point $x\in Z$ there is exactly one curve from $\cC$ containing $x$. Suppose the curves from $\cC$ are not contained in $Z$. Then the following holds:
\begin{enumerate}
\item If $Z$ is unirational, then $X$ is unirational.
\item If moreover $Z$ is rational and there exists open $V\subset Z$ such that any curve from $\cC$ intersects $V$ in no more than one point, then $X$ is rational.
\end{enumerate}
\end{lem}
\begin{proof}
Denote the total space of the family of curves also by $\cC$. It comes with maps $\pi:\cC\to S$ and $f:\cC\to X$. Let $L$ be the locus of points $x\in X$ such that there is exactly one curve from $\cC$ containing $x$. This is a constructible algebraic subset of $X$. By the assumptions, $\dim(X\setminus L)\leq n-1$. Therefore $\dim(\overline{(X\setminus L)}\setminus (X\setminus L))\leq n-2$. Let $U=X\setminus\overline{(X\setminus L)}$. Since $\dim(Z\cap L)=n-1$, we also have $\dim(Z\cap U)=n-1$. Let $s:U\to S$ be the algebraic map which sends a point $x\in U$ to the unique $s(x)\in S$ such that $x\in C_{s(x)}$. The pullback $s^* \cC$ of the original family of curves to $U\cap Z$ has a natural section: for any $x\in U\cap Z$ the curve $C_{s(x)}$ contains $x$. Therefore over a non-empty open subset $W\subset U\cap Z$ this family is trivial. We obtain a map
\[
f':W\times \PP^1 \subset s^*\cC \to \cC \to X.
\]

If $Z$ is unirational, then $W$ is unirational. Hence $W\times \PP^1$ is unirational. The image of $f'$ is irreducible and contains $W$. Thus it is either contained in $\overline{W}=Z$, or has dimension $n$. The former is not possible because curves from $\cC$ are not contained in $Z$. Thus the image of $f'$ has dimension $n$. Therefore $f'$ is dominant and $X$ is unirational. The first statement has been proved.

To prove the second statement, we assume without loss of generality that $W\subset V$. If $Z$ is rational, then $W$, and hence also $W\times \PP^1$ is rational. So it is enough to show that a generic point of $X$ has not more than $1$ preimage under $f'$. Suppose $x\in U$ has at least $2$ preimages. This means there are $(v_1, t_1), (v_2, t_2) \in W\times \PP^1$ that go to $x$. Since there is exactly one curve from $\cC$ passing through $x$, and that curve can intersect $W$ in at most one point, we obtain $v_1=v_2$. On the other hand, for each $v\in W$ there is at most finitely many values of $t$ such that there exist $t'$ such that $f'(v, t)=f'(v,t')$. So the dimension of such pairs $(v,t)$ is at most $n-1$, and therefore the dimension of the space of such $x$ is also at most $n-1$. So a generic $x$ has no more than $1$ preimage.
\end{proof}

Although the proof of Lemma \ref{lem: main lemma} is essentially trivial, we see that proving unirationality/rationality of $X$ is reduced to unirationality/rationality of $Z$, and a purely curve counting question. 

A similar idea appeared in \cite{ionescu2003rationality}, where the authors show that existence of a unique \emph{quasi-line} passing through \emph{two} general points implies rationality.
\subsection{Counting curves on a computer}
The families of curves we will be dealing with are such that one can write down explicitly a system of equations whose solutions correspond to curves passing through a given point. So we can implement the following strategy. Pick a big prime number, for instance $p=1000003$ or $p=1000033$. We will work over $F=\GF(p)$. Generate a random point $x\in X(F)$. Compute the number of curves passing through $x$ by counting solutions over $\bar F$ of the corresponding system of equations by the standard Gr\"obner basis techniques\footnote{In our computations we used SAGE (\cite{sagemath}), which delegates Gr\"obner basis computations to Singular (\cite{singular}) and certain lattice algorithms to GAP (\cite{GAP4})}. If this number is $k$, $p$ is large and $x$ is ``sufficiently random'', then we expect $x$ to behave like a generic point, so the number of curves for a generic point over the complex numbers should also be $k$.

\section{Rational curves}
There are exactly $3$ pairs $E, \Gamma$ where $E$ is an elliptic curve over $\CC$ and $\Gamma$ is a subgroup of the group of automorphisms of $E$ with $|\Gamma|>2$. Consider an elliptic curve $E$ of the form $x^2-y^3=z^6$ in $\PP(3,2,1)$ or $x^2-y^4=z^4$ in $\PP(2,1,1)$ or $x^3-y^3=z^3$ in $\PP(1,1,1)=\PP^2$. The equation of the curve in all cases is $x^a-y^b=z^c$ in $\PP(\tfrac{c}{a}, \tfrac{c}{b}, 1)$. We choose $(1,1,0)$ as the zero point on $E$. There are $\gcd(a,b)$ points with $z=0$, which we call ``points at infinity''. Let $\zeta$ be a primitive root of unity of order $c$. The group $\Gamma$ of the roots of unity of order $c$ acts on $E$ by 
\[
\zeta (x,y,z)=(x, y, \zeta z).
\]

We construct rational curves in $E^n/\Gamma$ as follows. Let $k\geq 1$ be an integer, and let $R(t,u)$ be a homogeneous polynomial of degree $ck$. For each $i=1,2,\ldots,n$ let $P_i(u,v)$, $Q_i(u,v)$ be relatively prime homogeneous polynomials of degrees $\tfrac{kc}{a}, \tfrac{kc}{b}$ respectively satisfying
\begin{equation}\label{eq:PQR}
P_i^a - Q_i^b = R.
\end{equation}
Let $\tilde{C}$ be the curve given by equation $R(u,v)=w^{c}$ in $\PP(1,1,k)$. The group $\Gamma$ acts on $\tilde{C}$ by $\zeta(u,v,c)=(u,v,\zeta c)$ and $\tilde{C}/\Gamma=\PP^1$. For each $i$ we have a $\Gamma$-equivariant map $f:\tilde{C}\to E$ by
\[
(u,v, w) \to (P_i(u,v), Q_i(u,v), w).
\]
Quotienting out by $\Gamma$ we obtain a commutative diagram
\[
\begin{tikzcd}
\tilde{C} \arrow{r}{f} \arrow{d} & E^n \arrow{d}\\
C\cong \PP^1 \arrow{r} & E^n/\Gamma.
\end{tikzcd}
\]

\subsection{Discrete invariants}\label{sec:discrete invs}
To every such curve we associate \emph{discrete invariants} as follows. For each $i\neq j$ we have
\[
P_i^a - Q_i^b = P_j^a - Q_j^b.
\]
Thus we have
\[
\prod_{l=0}^{a-1} (P_i-\zeta^{\frac{lc}{a}} P_j) = \prod_{r=0}^{b-1} (Q_i-\zeta^{\frac{rc}{b}} Q_j).
\]
Denote
\[
G_{i,j}^{l,r} = \gcd(P_i-\zeta^{\frac{lc}{a}} P_j, Q_i-\zeta^{\frac{rc}{b}} Q_j),\quad M_{i,j}^{l,r} = \deg G_{i,j}^{l,r} \qquad (0\leq l<a,\,0\leq r<b).
\]
Using the assumption that $P_i$ and $Q_i$ are relatively prime and considering contribution of an arbitrary linear form in $u$, $v$ to various $M_{i,j}^{l,r}$ we establish the following:
\[
\frac{kc}{a} = \deg (P_i-\zeta^{\frac{lc}{a}} P_j) = \sum_{r=0}^{b-1} M_{i,j}^{l,r},
\]
\[
\frac{kc}{b} = \deg (Q_i-\zeta^{\frac{rc}{a}} Q_j) = \sum_{l=0}^{a-1} M_{i,j}^{l,r}.
\]
Note that $\gcd(G_{i,j}^{l,r},G_{i,j}^{l',r'})=1$ whenever $l\neq l'$ and $r\neq r'$ because otherwise all the $4$ polynomials $P_i, P_j, Q_i, Q_j$ have a common divisor.

\subsection{Cohomology classes} It is useful to match the discrete invariants $M$ to the homology classes of the strict pullbacks of our curves in $H_2(\tilde{E^n/\Gamma},\ZZ)$, where $\tilde{E^n/\Gamma}$ is the blowup of $\E^n/\Gamma$ in the fixed points of $\Gamma$. It is possible to describe this homology group explicitly, but we will not do this. Instead we will think of the homology class of a rational curve as above consisting of two pieces of data:
\begin{enumerate}
\item The homology class of $\tilde{C}$ in $H_2(E^n,\ZZ)$.
\item For each $\Gamma$-fixed point $x\in E^n$ the intersection number of the strict pullback of $\tilde{C}$ to the blowup of $E_n$ in $x$ with the exceptional divisor. This, roughly speaking, counts how may points on $\tilde{C}$ go to $x$.
\end{enumerate}
Furthermore, the homology class of $\tilde{C}$ in $H_2(E^n,\ZZ)$ can be specified by the following data.
\begin{prop}
For each curve $\tilde{C} \subset E^n$ there exists a unique $n\times n$ Hermitian matrix $H(\tilde{C})$ with entries in $\QQ[\zeta]$ such that for any vector $v\in \ZZ[\zeta]^n$ we have
\[
v^* H(\tilde{C}) v = D_v\cdot \tilde{C},
\]
where $D_v$ is the divisor class given by the pullback of $0\in E$ to $E^n$ via the map $\pi_v:E^n\to E$ given by $(x_1,x_2,\ldots,x_n) \to \sum_i {v_i x_i}$, and $v^*$ denotes the conjugate transpose of $v$.
\end{prop}
\begin{proof}
It is well-known that the function $v\to D_v\cdot \tilde{C}$ is quadratic in $v$. Thus there exists a unique symmetric $\QQ$-bilinear form $B:\QQ[\zeta]^n \times \QQ[\zeta]^n \to \QQ$ such that $B(v, v)= D_v\cdot \tilde{C}$ for all $v$. But we have $D_{\zeta v} = D_{v}$. This implies $B(\zeta v, \zeta v) = B(v,v)$, hence $B(\zeta v, \zeta v')=B(v,v')$ for any pair of vectors $v,v'$. Let $H(\tilde{C}):\QQ[\zeta]^n\to \QQ[\zeta]^n$ be the unique $\QQ$-linear map such that
\[
B(v,v') = \Re(v^* H(\tilde{C}) v') \qquad\text{for all pairs $v, v'\in \QQ[\zeta]^n$.} 
\]
We have
\[
\Re(v^* H(\tilde{C})\zeta v') = B(v,\zeta v') = B(\overline{\zeta} v,v') = \Re(v^* \zeta H(\tilde{C}) v').
\]
Since this holds for all $v,v'$ the map $H(\tilde{C})$ must be $\QQ[\zeta]$-linear. So it can be represented by a matrix with entries in $\QQ[\zeta]$, and that matrix must be Hermitian because the form $B$ was symmetric.
\end{proof}

It is clear that the diagonal entries of $H(\tilde{C})$ are simply the degrees of the components $f_i$ of $f$, $f_i:\tilde{C}\to E$. Let us calculate the degree of these components for our construction. Consider the function
\[
\frac{x}{z^{\frac{c}{a}}}.
\]
This is a rational function of degree $b$ on $E$ because for a generic $t\in\CC$ there are exactly $b$ solutions to $\frac{x}{z^{\frac{c}{a}}}=t$ corresponding to the $b$-th roots of $t^a-1$. Its pullback to $\tilde{C}$ is the function
\[
\frac{P_i(u,v)}{w^{c/a}}.
\]
Now the equation $\frac{P_i(u,v)}{w^{c/a}}=t$ has $ck\cdot \frac{c}{a}$ solutions: $ck$ values of $u/v$ obtained by solving $P_i(u,v)^a=t R(u,v)$, and $c/a$ values of $w/v^k$ for each of these. Thus the degree of $f_i$ is
\[
\frac{ck \cdot \frac{c}{a}}{b} = \frac{k c^2}{ab}.
\]
A recipe to calculate the off-diagonal entries from the matrices $M_{i,j}$ will be given in the next section on a case-by-case basis.

\subsection{Calculating $k$}
Finally, we calculate the value of $k$ as a function of $n$ for which we expect to have finite number of our curves passing through a generic point of $E^n$. The first coefficient of $R(u,v)$ can be normalized to $1$, and we  have $kc$ remaining coefficients. A generic point is given by pairs $x_i, y_i$ satisfying $x_i^a-y_i^b=1$, and we can parametrize our curve so that the point $(u,v,w)=(1,0,1)$ goes to $(x_i, y_i, 1)$. This fixes the first coefficient of $P_i$ and $Q_i$. Then the condition for a polynomial $R$ to be of the form $P^a-Q^b$ is of codimension $k(c-\frac{c}{a}-\frac{c}{b})$. Thus the expected dimension of the space of solutions is $kc-n k(c-\frac{c}{a}-\frac{c}{b})$. We want this number to be equal to $2$ because there is a $2$-dimensional group of translations an rotations acting on solutions that needs to be gauged out. Thus we have
\[
2 = kc-n k (c-\frac{c}{a}-\frac{c}{b}).
\]
Note that we have $\frac{1}{a}+\frac{1}{b}+\frac{1}{c}=1$ in all the three cases, so we obtain
\[
k = \frac{2}{c-n}.
\]

\subsection{Summary of the approach}
We summarize our strategy for proving rationality of varieties of the form $E^n/\Gamma$ corresponding to triples $(a,b,c)=(3,3,3)$, $(a,b,c)=(2,4,4)$, $(a,b,c)=(2,3,6)$ and $n<c$. 
\begin{itemize}
\item Calculate $k=\frac{2}{c-n}$. Suppose it is an integer\footnote{The only cases with $n>1$ when this number is not an integer are $(a,b,c)=(2,3,6)$ with $n=2,3$. In these cases the method can still be applied. The curve $\tilde{C}$ should pass through $\Gamma$-fixed points of orders different from $6$, which implies a slightly different general shape of the equations \eqref{eq:PQR}. We do not include these situations here because it would complicate the notations, and because these cases are already known to be rational anyway.}. 
\item List possible $a\times b$ matrices $M$ and figure out which matrices correspond to which off-diagonal values of $H$. Obtain a list of possible off-diagonal entries $\mathbf{h}=\{h_1,h_2,\ldots, h_m\}$.
\item List possible $n\times n$ matrices $H$ up to integral change of basis which are positive-definite, have $\frac{kc^2}{ab}$ on the diagonal, and have only off-diagonal entries from the list $h_1,h_2,\ldots, h_m$.
\item For each $n\times n$ matrix $H$ list the degrees $M_{i,j}^{l,r}$.
\item For a point $p=(p_1,p_2,\ldots,p_n)\in E^n$, $p_i=(x_i,y_i,1)$ try to compute how many curves with discrete invariants $M_{i,j}^{l,r}$ pass through $p$. A curve is determined by a sequence of homogeneous polynomials $G_{i,j}^{l,r}(u,v)$ with first coefficient $1$ of degrees $M_{i,j}^{l,r}$. These polynomials must satisfy $\gcd(G_{i,j}^{l,r},G_{i,j}^{l',r'})=1$ whenever $l\neq l'$ and $r\neq r'$, and the equations obtained by elimination of $P_1,\ldots, P_n$ and $Q_1,\ldots,Q_n$ from the following ($i,j=1,\ldots,n$, $l=0,\ldots,a-1$, $r=0,\ldots,b-1$) system of \emph{main equations}:
\begin{equation}\label{eq:maineq}
P_i-\zeta^{\frac{lc}{a}} P_j = (x_i-\zeta^{\frac{lc}{a}} x_j) \prod_{r=0}^{b-1} G_{i,j}^{l,r},\quad Q_i-\zeta^{\frac{rc}{b}} Q_j = (y_i-\zeta^{\frac{rc}{b}} y_j) \prod_{l=0}^{a-1} G_{i,j}^{l,r}.
\end{equation}
\item If we are lucky and the answer to the previous step is $1$ for a generic point $p$, then try to construct a vector $v\in\ZZ[\zeta]^n$ such that for a generic point $p\in D_v$ the number of curves is also one, and the number of intersection points of $D_v/\Gamma\cap C$ outside the set of fixed points of $\Gamma$ is at most $1$.
\end{itemize}

\section{Example for $(a,b,c)=(3,3,3)$.}
In this case the group $\Gamma$ has order $c=3$, so we have only one case $n=2$, $k=2$. The discrete invariant has the form of a matrix
\[
M = \begin{pmatrix} M^{0,0} & M^{0,1} & M^{0,2}\\
 M^{1,0} & M^{1,1} & M^{1,2}\\
  M^{2,0} & M^{2,1} & M^{2,2}\end{pmatrix}
\]
of non-negative integers with all the row and column sums equal $2$. To calculate the $2\times 2$ matrix $H(\tilde{C})$ we already know that the diagonal entries are $2$. Let 
\[
H = \begin{pmatrix} 2 & h_{1,2}\\
 \overline{h_{1,2}} & 2\end{pmatrix}
\]
One can relate $h_{1,2}$ to $M$ by the following. Let $\Delta\subset E\times E$ be the diagonal. Then we have
\begin{prop}
\[
\Delta \cdot \tilde{C} = 2 + M^{0,0} + M^{1,2} + M^{2,1}.
\]
\end{prop}
\begin{proof}
Consider curves $E_{l,r}\subset E\times E$ defined by equations on $(x_i,y_i,z_i)\in E$ ($i=1,2$):
\[
\frac{x_1}{z_1} = \zeta^l \frac{x_2}{z_2},\qquad \frac{y_1}{z_1} = \zeta^r \frac{y_2}{z_2}.
\]
It turns out, that $E_{0,0}$, $E_{1,2}$, $E_{2,1}$ do not intersect. Therefore they have the same homology class. So, by counting the intersection points
\[
E_{0,0}\cdot \tilde{C} = \frac13 (E_{0,0} + E_{1,2} + E_{2,1}) \cdot \tilde{C} \geq M^{0,0}+M^{1,2}+M^{2,1} + 2.
\]
Analogously,
\[
E_{0,1} \cdot \tilde{C} \geq M^{0,1}+M^{1,0}+M^{2,2} + 2,
\]
\[
E_{0,2} \cdot \tilde{C} \geq M^{0,2}+M^{2,0}+M^{1,1} + 2.
\]
The divisor $[E_{0,0}]+[E_{0,1}]+[E_{0,2}]$ is linearly equivalent to $[E\times D] + [D\times E]$ where $D$ is the divisor at infinity of $E$, which has degree $3$. Thus we obtain
\[
(E_{0,0} + E_{0,1} +E_{0,2})\cdot\tilde{C} = 12.
\]
Hence the inequalities are equalities.
\end{proof}
The diagonal corresponds to the vector $(1,-1)$. This gives us
\[
4-h_{1,2}-\bar h_{1,2} = 2 + M^{0,0}+M^{1,2}+M^{2,1}.
\]
Similarly we obtain the evaluation for the vector $(\zeta,-1)$, which corresponds to the curve $E_{1,1}$:
\[
4 - \bar\zeta h_{1,2} - \zeta \bar h_{1,2} = 2 + M^{1,1}+M^{0,2}+M^{2,0},
\]
which allows to calculate $h_{1,2}$:
\[
h_{1,2} = 2 + 2\zeta + \frac{\zeta^2(M^{1,1}+M^{0,2}+M^{2,0}) - M^{0,0}-M^{1,2}-M^{2,1}}{1-\zeta}.
\]
Going over the set of possible $M$ we find the set of possible values of $h_{1,2}$:
\[
h_{1,2}\in \{0, 1, \zeta, \zeta^2, -1, -\zeta, -\zeta^2\}.
\]
Up to a integral change of basis (a matrix $g\in\GL_2(\ZZ[\zeta])$ sends $H$ to $g^* H g$) we have two possible matrices, with determinants $3$ and $4$:
\[
H_3 = \begin{pmatrix} 2 & -1\\-1 & 2 \end{pmatrix},\quad H_4 = \begin{pmatrix} 2 & 0\\0 & 2 \end{pmatrix}.
\]
For $H_3$ we have $3$ possible matrices $M$:
\[
M = \begin{pmatrix} 2 & 0 & 0\\
 0 & 1 & 1\\
  0 & 1 & 1\end{pmatrix},\quad
M = \begin{pmatrix} 1 & 0 & 1\\
 1 & 0 & 1\\
  0 & 2 & 0\end{pmatrix},\quad
M = \begin{pmatrix} 1 & 1 & 0\\
 0 & 0 & 2\\
  1 & 1 & 0\end{pmatrix}.
\]
For $H_4$ we have $6$ possible matrices $M$:
\[
\begin{pmatrix}
2 & 0 & 0 \\
0 & 2 & 0 \\
0 & 0 & 2
\end{pmatrix}
,\;
\begin{pmatrix}
0 & 0 & 2 \\
2 & 0 & 0 \\
0 & 2 & 0
\end{pmatrix}
,\;
\begin{pmatrix}
0 & 2 & 0 \\
0 & 0 & 2 \\
2 & 0 & 0
\end{pmatrix}
\]
\[
\begin{pmatrix}
1 & 1 & 0 \\
0 & 1 & 1 \\
1 & 0 & 1
\end{pmatrix}
,\;
\begin{pmatrix}
1 & 0 & 1 \\
1 & 1 & 0 \\
0 & 1 & 1
\end{pmatrix}
,\;
\begin{pmatrix}
0 & 1 & 1 \\
1 & 0 & 1 \\
1 & 1 & 0
\end{pmatrix}
\]
Some matrices do not produce any curves passing through generic points, for instance the first $3$ matrices corresponding to $H_4$. To illustrate our method we give here an explicit parametrization of the curves corresponding to $H_3$:
\subsection{$H_3$ curves}\label{sec:H3 curves}
It is enough to consider only the first matrix, because the other $2$ can be obtained from it by automorphisms:
\[
M = \begin{pmatrix} 2 & 0 & 0\\
 0 & 1 & 1\\
  0 & 1 & 1\end{pmatrix},\quad H=\begin{pmatrix}2 & -1\\-1 & 2\end{pmatrix}.
\]
We want to determine how many curves pass through a given point. Take $p_i=(x_i,y_i,1)$ for $i=1,2$ points on $E$. If a curve $\tilde{C}$ passes through $p=(p_1,p_2)$ then we can choose the coordinates $u,v$ such that $p$ is at $v=0$. We can still apply affine transformations $(u,v)\to (a u + b v, v)$. Such a curve is then completely determined by the homogeneous polynomials $G_{l,r}(u,v)$ of degrees $M^{l,r}$ with first coefficient $1$ satisfying the following equations, which follow from $\sum_{l=0}^2 \zeta^l (P_1 - \zeta^l P_2)=0$ and a similar equation for $Q$:
\[
\sum_{l=0}^2 \zeta^l (x_1-\zeta^l y_1) \prod_{r=0}^2 G_{l,r}=0,\quad
\sum_{r=0}^2 \zeta^l (x_2-\zeta^l y_2) \prod_{l=0}^2 G_{l,r}=0.
\]
In our situation, we have $1$ polynomial of degree $2$ and $4$ polynomials of degree $1$:
\[
(x_1-y_1) G_{0,0} + (\zeta x_1-\zeta^2 y_1) G_{1,1} G_{1,2} + (\zeta^2 x_1-\zeta y_1) G_{2,1} G_{2,2},
\]
\[
(x_2-y_2) G_{0,0} + (\zeta x_2-\zeta^2 y_2) G_{1,1} G_{2,1} + (\zeta^2 x_2-\zeta y_2) G_{1,2} G_{2,2}.
\]
 The polynomial of degree $2$ is $G_{0,0}$. Note that $x_1-y_1\neq 0$ and $x_2-y_2\neq 0$. So we can eliminate $G_{0,0}$ from the equations:
\[
G_{1,1}\left((x_2-y_2)(\zeta x_1-\zeta^2 y_1) G_{1,2}-(x_1-y_1)(\zeta x_2-\zeta^2 y_2) G_{2,1}\right),
\] 
\[
=G_{2,2}\left((x_1-y_1)(\zeta^2 x_2-\zeta y_2) G_{1,2}-(x_2-y_2)(\zeta^2 x_1-\zeta y_1) G_{2,1}\right).
\]
The polynomials $G_{1,1}$, $G_{2,2}$ must be relatively prime, for otherwise $P_1, Q_1, P_2, Q_2$ would all share a factor. This implies
\[
(2\zeta+1)(x_2 y_1-x_1 y_2) G_{1,1} = \left((x_1-y_1)(\zeta^2 x_2-\zeta y_2) G_{1,2}-(x_2-y_2)(\zeta^2 x_1-\zeta y_1) G_{2,1}\right),
\]
\[
(2\zeta+1)(x_2 y_1-x_1 y_2) G_{2,2} = \left((x_2-y_2)(\zeta x_1-\zeta^2 y_1) G_{1,2}-(x_1-y_1)(\zeta x_2-\zeta^2 y_2) G_{2,1}\right).
\]
Assume $x_2 y_1-x_1 y_2\neq 0$. Then we uniquely reconstruct $G_{1,1}, G_{2,2}$ from $G_{1,2}$, $G_{2,1}$. Again $G_{1,2}$, $G_{2,1}$ are relatively prime, and by applying affine transformations we can move them to an arbitrary pair of distinct linear polynomials with first coefficient $1$, for instance $u$, $u-v$. So under our assumptions there is at most one curve passing through $p$. Vice versa, to show that the curve exist we just need to make sure that in our construction the pairs $(G_{0,0}, G_{1,2})$, $(G_{0,0}, G_{2,1})$, $(G_{0,0}, G_{1,1})$, $(G_{0,0}, G_{2,2})$, $(G_{1,1}, G_{2,2})$, $(G_{1,2}, G_{2,1})$ are relatively prime. This requires another condition: $x_1 x_2 - y_1 y_2\neq 0$.

So we have shown that the curve is unique provided
\[
z_1\neq0,\quad z_2\neq0,\quad x_2 y_1-x_1 y_2\neq 0,\quad x_1 x_2 - y_1 y_2\neq 0.
\]
This means we have to remove the divisors given by vectors $(1,\pm \zeta^i)$, $(0,1)$, $(1,0)$. Taking any other divisor class we will satisfy conditions for part (i) of Lemma \ref{lem: main lemma}. To show rationality we need to satisfy the assumptions of part (ii). So we need a divisor with small intersection number with $\tilde{C}$, i.e. a vector not of the form $(\pm\zeta^i,0)$, $(0,\pm\zeta^i)$, $(\pm\zeta^i, \pm\zeta^j)$ whose length is small with respect to the form $H$. Take $v=(1, 2+\zeta)$, which corresponds to the divisor $D_v$ consisting of $(p_1,p_2)\in E^2$ such that $p_1 + 2 p_2 + \zeta p_2=0$. We have $v^* H v=5$. So there is at most $5$ points of intersection in $D_v\cap \tilde{C}$. Going down to $E^2/\Gamma$ we obtain at most $\lfloor\frac{5}{3}\rfloor = 1$ of points of intersection $(D_v/\Gamma)\cap C$ satisfying $z_i\neq 0$. Clearly, $D_v/\Gamma$ is rational. So the conditions of Lemma \ref{lem: main lemma} are satisfied.

\subsection{$H_4$ curves}
In this case computer experiments showed that there are $3$ curves passing through a generic point for each of the last $3$ matrices $M$. However these curves can be distinguished by their incidence information with the $\Gamma$-fixed points, so probably it is possible to use these curves for an alternative rationality proof.

\subsection{Total curve count}
In total we obtain $3$ curves for $H_3$ and $9$ curves for $H_4$. However, these curves can be distinguished by our discrete invariants and by their intersections with $z_1=z_2=0$.

\section{Examples for $(a,b,c)=(2,4,4)$}
If $(a,b,c)=(2,4,4)$, we can have $n=2$ or $n=3$. Here $\zeta=\sqrt{-1}$. For $n=2$ we obtain $k=1$. For $n=3$ we obtain $k=2$. The matrices $M_{i,j}$ are $2\times 4$ with column sums $k$ and row sums $2k$. The matrices $H$ have $2k$ on the diagonal.

\begin{prop}\label{prop: 244} The intersection number of the diagonal $\Delta\subset E\times E$ and $\tilde{C}$ is given by
\[
\Delta \cdot \tilde{C} = 2k + 2 M^{0,0} + 2 M^{1,0}.
\]
\end{prop}
\begin{proof}
We have curves $E_{l,r}\subset E\times E$ given by equations ($l=0,1$, $r=0,1,2,3$)
\[
\frac{x_1}{z_1^2} = (-1)^l \frac{x_2}{z_2^2},\qquad \frac{y_1}{z_1} = \zeta^r \frac{y_2}{z_2}.
\]
The pairs representing the same homology class are listed as follows $(E_{0,0}, E_{1,2})$, $(E_{0,1}, E_{1,3})$, $(E_{0,2}, E_{1,0})$, $(E_{0,3}, E_{1,1})$. So
\[
E_{0,0}\cdot \tilde{C} = \frac{1}{2}(E_{0,0}+E_{1,2})\cdot \tilde{C} \geq 2 M^{0,0}+2 M^{1,2} + 2k.
\]
This is because each root of $\gcd(P_1-P_2,Q_1-Q_2)$ has multiplicity $4$ in $E_{0,0}\cdot \tilde{C}$, and there are further $k c=4k$ points with $w=0$ on $\tilde{C}$ which map to the points with $z_1=z_2=0$. Producing similar inequality for $E_{1,0}$ and adding to the one above we obtain 
\[
(E_{1,0}+E_{0,0}) \cdot \tilde{C} \geq 8k.
\] 
On the other hand, $E_{1,0}+E_{0,0}$ is equivalent to $E\times D + D\times E$, where $D$ is the divisor at infinity consisting of $2$ points. So the intersection equals $8k$. Therefore our inequalities must be equalities.
\end{proof}
This allows us to compute $h_{i,j}$ as a function of the entries of $M_{i,j}$. The diagonal corresponds to the vector $e_i-e_j$, so we have
\[
4k - h_{i,j} - \bar h_{i,j} = 2k + 2 M_{i,j}^{0,0} + 2 M_{i,j}^{1,2}.
\]
Hence $\Re h_{i,j} = k-M_{i,j}^{0,0}-M_{i,j}^{1,2}$. The vector $e_i-\zeta e_j$ corresponds to the curve $E_{1,3}$, so the corresponding intersection number is 
\[
4k - \zeta h_{i,j} - \bar \zeta\bar h_{i,j} = 2k + 2 M_{i,j}^{1,3} + 2 M_{i,j}^{0,1}.
\]
We obtain $\Im h_{i,j} = -k + M_{i,j}^{0,1} + M_{i,j}^{1,3}$. Thus
\[
h_{i,j} = k(1-\zeta) - M_{i,j}^{0,0} - M_{i,j}^{1,2} + \zeta M_{i,j}^{0,1} + \zeta M_{i,j}^{1,3}.
\]

\subsection{The case $k=1$, $n=2$}
There are two $H$-matrices (up to automorphisms) for $n=2$, $k=1$, of determinants $2$ and $4$:
\[
H_2=\begin{pmatrix}2 & \zeta-1 \\ -\zeta-1 & 2\end{pmatrix},\quad
H_4=\begin{pmatrix}2 & 0 \\ 0 & 2\end{pmatrix}.
\]
For $H_2$ there is only $1$ matrix $M$:
\[
M = \begin{pmatrix} 1 & 1 & 0 & 0\\
 0 & 0 & 1 & 1\end{pmatrix}.
\]
For $H_4$ there are $2$ matrices:
\[
M = \begin{pmatrix} 1 & 0 & 1 & 0 \\
0 & 1 & 0 & 1\end{pmatrix},\quad
M = \begin{pmatrix}0 & 1 & 0 & 1 \\
1 & 0 & 1 & 0
\end{pmatrix}.
\]
It is not so difficult to check that each of the $3$ matrices $M$ leads to a good family of curves.

\subsection{The case $k=2$, $n=3$}\label{sec:k2n3}
With $n=3$ the set of possibilities is much bigger. We have $19$ possible matrices $M$. They produce the following list of $13$ possible off-diagonal entries for $H$:
\[
0,\; \pm 2,\; \pm 2\zeta,\; \pm 2 \zeta \pm 2,\; \pm \zeta \pm 1.
\]

To construct a curve we need to choose $3$ of them to get $M_{i,j}$ with $(i,j)=(1,2),(1,3),(2,3)$. So there are $19^3=6859$ possibilities. Classifying all the possible positive definite $3\times 3$ matrices $H$ up to $\GL_3(\ZZ[\zeta])$ action produces $14$ cases with determinants
\[
8, 16, 16, 24, 32, 32, 32, 36, 40, 44, 48, 48, 56, 64.
\]
Counting curves on a computer produces Table \ref{table: 244}. 

\begin{table}
{\renewcommand{\arraystretch}{1.2}
\begin{tabular}{c|c|c|c|c|c|c|c|c|c|c}
$\det(H)$ & 8 & 16 & 24 & 32 & 36 & 40 & 44 & 48 & 56 & 64\\
\hline
\vspace{2mm} $\# C$ & 0 & 0,1 & 4 & 8,10,18 & 8 & 20 & 24 & 48,60 & 80 & 212\\
\end{tabular}
}
\caption{Curve counts for $(a,b,c)=(2,4,4)$, $n=3$, $k=2$.}
\label{table: 244}
\end{table}

The $H$-matrices with $0$ curves are the following matrices with determinants $8$ resp. $16$:
\[
\begin{pmatrix}
4 & -2 \zeta - 2 & -2 \\
2 \zeta - 2 & 4 & -\zeta - 1 \\
-2 & \zeta - 1 & 4
\end{pmatrix},\quad
\begin{pmatrix}
4 & -2 & -2 \\
-2 & 4 & -\zeta - 1 \\
-2 & \zeta - 1 & 4
\end{pmatrix}.
\]
It turns out that non-existence of these curves is explained by the fact that the matrices can be conjugated to 
\[
\begin{pmatrix}
4 & 2 & 0 \\
2 & 4 & -\zeta + 3 \\
0 & \zeta + 3 & 4
\end{pmatrix},\quad
\begin{pmatrix}
4 & 2 & \zeta + 3 \\
2 & 4 & \zeta + 1 \\
-\zeta + 3 & -\zeta + 1 & 4
\end{pmatrix},
\]
which contain forbidden off-diagonal entries $\zeta + 3$.

Note that for each matrix $H$ there are several triples of matrices $M_{i,j}$. The table was obtained by adding the point counts for all triples. In some situations the total number of curves can be greater than $1$, but for some individual triples $M_{i,j}$ the number is $1$. We will work with the matrix of determinant $16$ which gives $1$ curve. The matrix is
\[
H_{16}=\begin{pmatrix}
4 & 2 & 2 \zeta \\
2 & 4 & 2 \\
-2 \zeta & 2 & 4
\end{pmatrix}.
\]
The matrices $M_{i,j}$ are as follows:
\[
M_{1,2}=\begin{pmatrix}
0 & 1 & 2 & 1 \\
2 & 1 & 0 & 1
\end{pmatrix},\quad
M_{1,3}=\begin{pmatrix}
1 & 2 & 1 & 0 \\
1 & 0 & 1 & 2
\end{pmatrix},\quad
M_{2,3}=\begin{pmatrix}
0 & 1 & 2 & 1 \\
2 & 1 & 0 & 1
\end{pmatrix}.
\]

%We proceed parametrizing the curves as in \ref{sec:H3 curves}. The following homogeneous polynomials in $u,v$ have degree $1$:
%\[
%G_{1,2}^{0,1}, G_{1,2}^{1,1}, G_{1,2}^{0,3}, G_{1,2}^{1,3},
%\]
%\[
%G_{1,3}^{0,0}, G_{1,3}^{1,0}, G_{1,3}^{0,2}, G_{1,3}^{1,2},
%\]
%\[
%G_{2,3}^{0,0}, G_{2,3}^{1,0}, G_{2,3}^{0,2}, G_{2,3}^{1,2}.
%\]
%The unknown polynomials of degree $2$ can be eliminated.
Computer experiments show that exactly one curve passes through a generic point of $E^3$. To apply Lemma \ref{lem: main lemma} in full generality we need to choose a divisor. So we look for a vector $v$ whose $H$-norm $v^* H v$ is small, but not too small. All vectors of norm $4$ do not produce good divisors: through a generic point of such divisor there are no curves of our type. There are no vectors of norms between $4$ and $8$. There are $252$ vectors of norm $8$. Let $\Aut(H)$ be the group of matrices $g\in \GL_3(\ZZ[\zeta])$ such that $g^* H g = H$. The vectors of norm $8$ form $3$ $\Aut(H)$-orbits. Some of these vectors are also such that through a generic point of the corresponding divisor there are no curves. In the orbit of $v=(1,-2,1)$, which consists of $192$ vectors, for $168$ vectors\footnote{Elements of $\Aut(H)$ acting on $E^3$ do not change $H$, but they still permute the $8$ points at infinity. So the true symmetry group of the system is not $\Aut(H)$, but a certain congruence subgroup. This explains why we obtain different curve counts for vectors of the same $\Aut(H)$-orbit.} the curve count is $1$ and for the remaining $24$ it is zero. This vector produces a divisor $D_v/\Gamma$ satisfying the conditions of Lemma \ref{lem: main lemma}. We have $D_v \cdot \tilde{C}=8$, so if we show that at least one intersection point is at infinity, we obtain the number of finite intersection points of $D_v/\Gamma$ with $C$ is at most $\lfloor \frac{7}{4} \rfloor=1$. The points at infinity of $D_v$ are $4$ points out of the total $2^3=8$ points at infinity on $E^3$. These are the points $(p_1,p_2,p_3)$ satisfying $p_1-2 p_2+p_3=0$. The points at infinity are of order $2$, so this condition is equivalent to $p_1=p_3$. Let $C'$ be the projection of $\tilde{C}\subset E\times E \times E$ to $E\times E$ using coordinates $1, 3$. So it is enough to show that $C'$ intersects $\Delta$ at infinity. The intersection number is $8$, but there are only $4$ finite intersection points because $M_{1,3}^{0,0}=1$. Thus there must be intersections at infinity.

\section{Examples for $(a,b,c)=(2,3,6)$}
Finally, we turn to the most interesting example, which includes open cases. We have $(a,b,c)=(2,3,6)$ and $n=4$ or $n=5$. Here $\zeta=e^{\frac{2\pi i}{6}}$. For $n=4$ we obtain $k=1$. For $n=5$ we obtain $k=2$. The matrices $M_{i,j}$ are $2\times 3$ with column sums $2k$ and row sums $3k$. The matrices $H$ have $6k$ on the diagonal. Some things are simpler because there is only $1$ point at infinity, and the correspondence between $M$-matrices and the off-diagonal entries of the $H$-matrix are bijective.

\begin{prop}\label{prop: 236} The intersection number of the diagonal $\Delta\subset E\times E$ and $\tilde{C}$ is given by
\[
\Delta \cdot \tilde{C} = 6k + 6 M^{0,0}.
\]
\end{prop}
\begin{proof}
We have curves $E_{l,r}\subset E\times E$ given by equations ($l=0,1$, $r=0,1,2$)
\[
\frac{x_1}{z_1^3} = (-1)^l \frac{x_2}{z_2^3},\qquad \frac{y_1}{z_1^2} = \zeta^{2r} \frac{y_2}{z_2^2}.
\]
We have
\[
E_{0,0}\cdot \tilde{C} \geq 6 M^{0,0} + 6k.
\]
This is because each root of $\gcd(P_1-P_2,Q_1-Q_2)$ has multiplicity $6$ in $E_{0,0}\cdot \tilde{C}$, and there are further $k c=6k$ points with $w=0$ on $\tilde{C}$ which map to the points with $z_1=z_2=0$. Producing similar inequality for $E_{1,0}$ and adding to the one above we obtain 
\[
(E_{1,0}+E_{0,0}) \cdot \tilde{C} \geq 24 k.
\] 
On the other hand, the divisor of the rational function $\frac{y_1}{z_1^2}-\frac{y_2}{z_2^2}$ is 
\[
E_{1,0}+E_{0,0}- 2(E\times D + D\times E),
\]
where $D$ is the point at infinity. Therefore
\[
(E_{1,0}+E_{0,0}) \cdot \tilde{C} = 2(E\times D + D\times E)\cdot \tilde{C}=24 k.
\]
Therefore our inequalities must be equalities.
\end{proof}

This allows us to compute $h_{i,j}$ as a function of the entries of $M_{i,j}$. The diagonal corresponds to the vector $e_i-e_j$, so we have
\[
12k - h_{i,j} - \bar h_{i,j} = 6k + 6 M_{i,j}^{0,0}.
\]
The vector $e_i-\zeta e_j$ corresponds to the curve $E_{1,2}$, so the corresponding intersection number is 
\[
12k - \zeta h_{i,j} - \bar \zeta\bar h_{i,j} = 6k + 6 M_{i,j}^{1,2}.
\]
So we can recover $h_{i,j}$:
\[
h_{i,j} = (4-2\zeta) k - (2+2\zeta) M_{i,j}^{0,0} + (4\zeta-2) M_{i,j}^{1,2}.
\]

\subsection{The case $k=1$, $n=4$}\label{sec:k1n4}
The diagonal entries of $H$ are $6$ and the possible off-diagonal entries are in the set
\[
\mathbf{h}=\{0,\; 2 \zeta - 4,\; 2 \zeta + 2,\; 4 \zeta - 2,\; -4 \zeta + 2,\; -2 \zeta - 2,\; -2 \zeta + 4\}.
\]
We classified all matrices $H$ up to $GL_4(\ZZ[\zeta])$-equivalence satisfying the following conditions:
\begin{enumerate}
\item $H$ is positive definite.
\item $H_{i,i}=6$ for $i=1,2,3,4$.
\item There is no vector $v\in\ZZ[\zeta]^4$ such that $v^* H v<6$.
\item For any $v_1, v_2\in\ZZ[\zeta]^4$ such that $v_i^* H v_i=6$ we have $v_1 H v_2 \in \mathbf{h}$. 
\end{enumerate}

It turns out there are $5$ matrices with determinants $144, 432, 576, 864, 1296$:
\[
H_{144}=\begin{pmatrix}
6 & 2 \zeta - 4 & 0 & 0 \\
-2 \zeta - 2 & 6 & 2 \zeta - 4 & 0 \\
0 & -2 \zeta - 2 & 6 & 2 \zeta - 4 \\
0 & 0 & -2 \zeta - 2 & 6
\end{pmatrix},
\]
\[
H_{432}=
\begin{pmatrix}
6 & 2 \zeta - 4 & 0 & 0 \\
-2 \zeta - 2 & 6 & 2 \zeta - 4 & 0 \\
0 & -2 \zeta - 2 & 6 & 0 \\
0 & 0 & 0 & 6
\end{pmatrix},
\]
\[
H_{576}=
\begin{pmatrix}
6 & 2 \zeta - 4 & 0 & 0 \\
-2 \zeta - 2 & 6 & 0 & 0 \\
0 & 0 & 6 & 2 \zeta - 4 \\
0 & 0 & -2 \zeta - 2 & 6
\end{pmatrix},
\]
\[
H_{864}=
\begin{pmatrix}
6 & 2 \zeta - 4 & 0 & 0 \\
-2 \zeta - 2 & 6 & 0 & 0 \\
0 & 0 & 6 & 0 \\
0 & 0 & 0 & 6
\end{pmatrix},
\quad
H_{1296}=
\begin{pmatrix}
6 & 0 & 0 & 0 \\
0 & 6 & 0 & 0 \\
0 & 0 & 6 & 0 \\
0 & 0 & 0 & 6
\end{pmatrix},
\]
Note that the off-diagonal values $0$ resp. $2\zeta-4$ correspond to $M_{i,j}=\begin{pmatrix} 1 & 1 & 1\\ 1& 1 & 1\end{pmatrix}$, $M_{i,j}=\begin{pmatrix} 2 & 1 & 0\\ 0& 1 & 2\end{pmatrix}$.
The curve counts are given in Table \ref{table: 236 n4}. It is not clear why curves corresponding to $H_{864}$ do not pass through generic points.
\begin{table}
{\renewcommand{\arraystretch}{1.2}
\begin{tabular}{c|c|c|c|c|c}
$\det(H)$ & 144 & 432 & 576 & 864 & 1296\\
\hline
$\# C$ & 1 & 6 & 12 & 0 & 72\\
\end{tabular}
}
\caption{Curve counts for $(a,b,c)=(2,3,6)$, $n=4$, $k=1$.}
\label{table: 236 n4}
\end{table}

We turn our attention to the matrix $H=H_{144}$, which already implies unirationality of $X_{4,6}$ and will also imply rationality if we find a ``good'' divisor class. The group 
\[
\Aut(H) = \{g\in\GL_4(\ZZ(\zeta))\,|\, g^* H g = H\}
\]
has order $155520$ and acts transitively on the $240$ vectors of $H$-norm $6$ and on the $2160$ vectors of $H$-norm $12$. Vectors of norm $6$ intersect $C$ only at infinity, so we pick a vector of norm $12$. Some of the vectors of norm $12$ correspond to the ``diagonals'', for instance $v=(1,0,1,0)$. For this vector we obtained $0$ curves. However picking $v=(0,1,2,1)$, and any other vector not of the form $(0,0,\zeta^i, \zeta^j)$ for some $i,j$ or a permutation of such, we obtain $1$ curve.

 Note that for any $v$ of norm $12$ and any curve $C$ of our kind the number of intersection points $\#(C\cap D_v/\Gamma)$ outside of the $\Gamma$-fixed points is at most $1$. This is true because $\tilde{C}\cdot D_v=12$, and the intersection $D_v\cap C$ contains at least one point at infinity.

So we make the following Conjecture, which by Lemma \ref{lem: main lemma} implies rationality of $X_{4,6}$:
\begin{conjecture}\label{conj:X46}
For $p\in E^4$ denote by $\#(p)$ the number of curves $\tilde{C}$ of our type corresponding to the matrix $H_{144}$ and containing $p$. Then for a generic point $p\in E^4$ we have $\#(p)=1$. Moreover, for a generic point $p\in D_{0,1,2,1}$ we have $\#(p)=1$, where
\[
D_{0,1,2,1} = \{(p_1,p_2,p_3,p_4)\in E^4\,|\, p_2 + 2 p_3 + p_4=0\}.
\]
\end{conjecture}
We verified this conjecture by testing the statement on $10000$ random points on $D_{0,1,2,1}$ and $10000$ random points on $E^4$ over the field $\GF_{1000003}$. Only $1$ point got ``unlucky'' and the number of curves was $0$. For every other point the number was $1$. Counting the curves took $\approx 0.05$ seconds per point on an ordinary laptop.

\subsection{The case $k=2$, $n=5$}
Finally we turn to the most interesting case. The diagonal entries of $H$ are $12$ and the possible off-diagonal entries are in the set
\[
\mathbf{h}=\{0,\, -4 \zeta + 8,\, 2 \zeta - 4,\, 6,\, 4 \zeta - 8,\, -2 \zeta + 4,\, 2 \zeta + 2,\, -4 \zeta - 4,\, -4 \zeta + 2,\, 4 \zeta - 2,
\]
\[
-6 \zeta,\, 6 \zeta - 6,\, 4 \zeta + 4,\, -6 \zeta + 6,\, -2 \zeta - 2,\, 6 \zeta,\, -8 \zeta + 4,\, -6,\, 8 \zeta - 4\}.
\]
We could not classify all such matrices $H$ up to $\GL_5(\ZZ[\zeta])$-equivalence because the set of possibilities is too big. However the following matrix seems to be the only matrix up to $\GL_5(\ZZ[\zeta])$-equivalence of the smallest possible determinant $24^3=13824$.
\[
H_{13824} = \begin{pmatrix}
12 & 4\zeta + 4 & 4\zeta + 4& 4\zeta + 4 & 4\zeta + 4\\
-4\zeta+8 & 12 & 4\zeta + 4 & 4\zeta + 4 & 4\zeta + 4\\
-4\zeta+8 & -4\zeta+8 & 12 & 4\zeta + 4 & 6\\
-4\zeta+8 & -4\zeta+8 & -4\zeta+8 & 12 & 6\\
-4\zeta+8 & -4\zeta+8 & 6 & 6 & 12\\
\end{pmatrix}
\]
We consider $H=H_{13824}$. Note that the off-diagonal value $H_{i,j}=4\zeta + 4$ resp. $H_{i,j}=6$ corresponds to $M_{i,j}=\begin{pmatrix}0&4&2\\4&0&2\end{pmatrix}$ resp. $M_{i,j}=\begin{pmatrix}0&3&3\\4&1&1\end{pmatrix}$.
There are exactly $336$ vectors of $H$-norm $12$. Since every curve $\tilde C$ has $12$ points at infinity, these curves cannot pass through generic points of divisors corresponding to these vectors. Just for reference we mention that the size of the group $\Aut(H)=\{g\in \GL_5(\ZZ[\zeta]) \,|\, g^* H g = H\}$ is $6912$. The vectors of $H$-norm $12$ form $3$ orbits of sizes $48$, $192$, $96$, represented by the basis vectors $e_1$, $e_3$, $e_5$. The next possible $H$-norm is $18$, and there are $768$ vectors of norm $18$ forming a single $\Aut(H)$-orbit. For such a vector $v$ we have  $D_v\cdot \tilde{C}=18$, and at least $12$ points of intersection are at infinity. Therefore $|D_v/\Gamma \cap C|\leq 1$. Some vectors represent ``generalized diagonals'', for instance $(0,0,1,0,-\zeta)$. We found that our curves do not pass through generic points on the corresponding divisors. Taking any vector different from those do seem to produce good divisors, for instance we take $v=(1,0,\zeta, 0,-1)$. 

The following conjecture implies unirationality of $X_{5,6}$ by part (i), Lemma \ref{lem: main lemma}:
\begin{conjecture}\label{conj:X56 1}
For $x\in E^5$ denote by $\#(x)$ the number of curves $\tilde{C}$ of our type corresponding to the matrix $H_{13824}$ and containing $x$. Then for a generic point $x\in E^5$ we have $\#(x)=1$.
\end{conjecture}

The following conjecture together with rationality of $X_{4,6}$ implies rationality of $X_{5,6}$ by Lemma \ref{lem: main lemma}:
\begin{conjecture}\label{conj:X56 2}
With $\#(p)$ defined in Conjecture \ref{conj:X56 1}, for a generic point $p\in D_{1,0,\zeta,0,-1}$ we have $\#(p)=1$, where
\[
D_{1,0,\zeta,0,-1} = \{(p_1,p_2,p_3,p_4,p_5)\in E^5\,|\, p_1 + \zeta p_3 = p_5\}.
\]
\end{conjecture}

\subsection{Computations for $n=5$}
The computations in these cases take much more time than in the $n=4$ case. It can probably be explained by the fact that the set of divisors where the number of curves is not $1$ is huge: for instance, it must contain all the $336$ divisors $D_v$ corresponding to vectors $v$ of $H$-norm $12$. Another issue is that when we create the ideal parametrizing our curves, we have besides equations also inequalities of the form
\begin{equation}\label{eq: extra ineq PQ}
\mathrm{resultant}\,(P_i, Q_i)\neq 0\qquad (1\leq i\leq 5).
\end{equation}
Each inequality is imposed by adding an extra variable $J_i$ and an extra equation
\begin{equation}\label{eq: extra eqn PQ}
J_i\;\mathrm{resultant}\,(P_i, Q_i) = 1\qquad(1\leq i\leq 5).
\end{equation}
Note that the degrees of $P_i$ and $Q_i$ are $6$ and $4$ respectively, so the resultant has degree $24$ and these extra equations are very long. On the other hand, when we tried to keep only the equations without the inequalities the length of the scheme of solutions grew up to $99$. The scheme turned out to contain a single isolated point and several very fat points failing the conditions $\gcd(P_i, Q_i)=1$. 

The computation with the inequalities \eqref{eq: extra ineq PQ} takes $\approx$ 1 hour 15 minutes on an ordinary laptop (for each random point on $E^5$). It turns out, it is better to extend the set of inequalities that translate to equations \eqref{eq: extra eqn PQ} by a much larger set of $32$ inequalities
\begin{equation}\label{eq: extra ineq GG}
\mathrm{resultant}\,(G_{i,j}^{l,r}, G_{i,j}^{l',r'})\neq 0\qquad \begin{matrix} (1\leq i<j\leq 5,\; 0\leq l<l'\leq 1,\; 0\leq r, r'\leq 2\;:\\\; r\neq r',\;M_{i,j}^{l,r}\neq 0,\; M_{i,j}^{l',r'}\neq 0).\end{matrix}.
\end{equation}
For each inequality we have to create a new variable and a new equation as in \eqref{eq: extra eqn PQ}. These inequalities formally follow from \eqref{eq: extra ineq PQ} as explained in Section \ref{sec:discrete invs}, but their degrees are much smaller. On the other hand, inequalities \ref{eq: extra ineq GG} do not seem to imply \ref{eq: extra ineq PQ}. Thus we must additionally test that every solution we find satisfies \ref{eq: extra ineq PQ}.

It turns out, that it is faster to build the ideal step-by step. On each \emph{step} we add some new equations and recompute the Gr\"obner basis. In the very beginning we choose a cell of the cell decomposition of the weighted projective space we do computations in. The total number of variables is $120$ (we have $10$ pairs $1\leq i<j\leq 5$ and for each pair $i,j$ we have $6$ polynomials $G_{i,j}^{l,r}$ whose degrees are given by the entries of $M_{i,j}$). Among these variables $42$ have weight $1$, $38$ have weight $2$, $22$ have weight $3$ and $18$ have weight $4$. We order the variables by weight, and if the weights agree by the degree of the polynomial they are coefficients of. Because we should consider the solutions up to translation, we set the very first variable to $0$. The choice of a cell in the weighted projective space means we set the first $r$ variables to $0$, the $r+1$-st variable to $1$. We need to do this for every $r$, $1\leq r \leq 119$. Then we have $3$ \emph{steps} (for each $r$):
\begin{enumerate}
\item Add equations coming from elimination of $P_i, Q_i$ from the main equations \eqref{eq:maineq}.
\item Add variables and equations representing \eqref{eq: extra ineq GG}.
\item Add variables and equations representing \eqref{eq: extra ineq PQ}.
\end{enumerate}
Then we compute the dimension over the base field of the quotient ring with respect to the ideal obtained in the final \emph{step}. This number divided by the weight of the variable we made equal to $1$ is the number of points in the given cell. If after some step we obtain ideal $(1)$, this means there is no solutions in a given cell, so we abort and pass to the next cell, i.e. next value of $r$. In all situations we encountered, all the solutions belonged to the biggest cell.

Complete computation for each point $p\in E^5(\GF_{1000003})$ took $\approx 6$ minutes on an ordinary laptop. We made $10$ trials for each of the Conjectures \ref{conj:X56 1}, \ref{conj:X56 2} and obtained exactly $1$ curve in all cases.

\begin{rem}
The quadratic form induced by $H_{13824}$ on the rank $10$ lattice $\ZZ[\zeta]^5$ is proportional to the so-called laminated lattice $\Lambda_{10}$, see \cite{conway1993sphere}. We discovered this fact with the help of OEIS (\cite{oeis}, sequence A006909) by searching for the sequence of numbers of vectors of given norm, which begins as follows: $1, 0, 336, 768$. In fact, the matrix $H_{144}$ from Section \ref{sec:k1n4} in a similar way corresponds to the lattice $E_8$. The matrix $H_{16}$ from Section \ref{sec:k2n3} corresponds to the lattice $D_6$.
\end{rem}

\section*{Acknowledgements}
I would like to thank F. Catanese for giving a wonderful talk about Ueno-Campana varieties, and K. Oguiso for a stimulating discussion and encouragement. I thank G. Williamson for correcting the statement of Corollary.

Initial results of this work were obtained in 2014 during my stay at ICTP, Trieste. The work was completed in 2017 during my stay at IST Austria, where I was supported by the Advanced Grant ``Arithmetic and Physics of Higgs moduli spaces'' No. 320593 of the European Research Council.

\bibliographystyle{amsalpha}
\bibliography{refs}

\end{document}